\documentclass
{amsart}
\usepackage{amssymb,amsmath,amsfonts,amsthm,amscd,enumerate}

\usepackage{graphics}
\usepackage[colorlinks=true]{hyperref}
\usepackage{amssymb,amsmath,amsfonts,amsthm,amscd,mathrsfs}

\input xy
\xyoption{all}

\newcommand{\CC}{\mathbb C}
\newcommand{\RR}{\mathbb R}

\newcommand{\maD}{\mathcal{D}}

\newcommand{\ZZ}{\mathbb{Z}}
\newcommand{\Sch}{\mathscr{S}}
\newcommand{\NN}{\mathbb{N}}

\newcommand{\maE}{\mathscr{E}}

\newcommand{\maG}{\mathcal{G}}

\newcommand{\smooth}{\mathcal{C}^{\infty}}

\newcommand{\ep}{\varepsilon}

\newcommand{\ip}[1]{\langle #1 \rangle}

\theoremstyle{definition}
\newtheorem{definition}{Definition}[section]

\theoremstyle{plain}
\newtheorem{theo}[definition]{Theorem}
\newtheorem{prop}[definition]{Proposition}
\newtheorem{lem}[definition]{Lemma}

\theoremstyle{remark}
\newtheorem{remark}[definition]{Remark}
\newtheorem{eg}[definition]{Example}

\newcommand{\R}{\mathbb R}
\newcommand{\N}{\mathbb N}

\newcommand{\Dom}{\mathrm{Dom}}
\newcommand{\eps}{\varepsilon}
\newcommand{\supp}{\mbox{supp}}
\newcommand{\D}{\mathcal{D}}
\newcommand{\cinfty}{{\mathcal C}^\infty}
\newcommand{\vphi}{\varphi}
\newcommand{\comp}{\subset\subset}
\newcommand{\gs}{{\mathcal G}}
\newcommand{\ns}{{\mathcal N}}
\newcommand{\esm}{{\mathcal E}_M}
\newcommand{\maS}{\mathcal{S}}
\newcommand{\hscomp}{H^s_{\mathrm{cp}}}
\newcommand{\hsncomp}{H^{s_0}_{\mathrm{cp}}}
\newcommand{\teps}{T_\eps(\sqrt{\Delta})}
\newcommand{\lt}[1]{\|#1\|_{L^2(M)}}
\newcommand{\WF}{\mathrm{WF}}
\newcommand{\WFg}{\mathrm{WF}_g}
\begin{document}
\title[Optimal regularization processes]{Optimal regularization processes on complete Riemannian manifolds}

\author[S. Dave]{Shantanu Dave}
\address{University of Vienna, Austria }
\email{shantanu.dave@unvie.ac.at}

\author[G. H\"ormann]{G\"unther H\"ormann}
\address{University of Vienna, Austria }
\email{guenther.hoermann@unvie.ac.at}

\author[M. Kunzinger]{Michael Kunzinger}
\address{University of Vienna, Austria }
\email{michael.kunzinger@unvie.ac.at}
\thanks{Supported by FWF grants P 20525-N13 and Y237-N13 of the Austrian Science Fund.}

\begin{abstract}
We study regularizations of Schwartz distributions on a complete Riemannian manifold $M$.
These approximations are based on families of smoothing operators obtained from 
the solution operator to the wave equation on $M$ derived from the metric Laplacian.
The resulting global regularization processes are optimal in the sense that they preserve
the microlocal structure of distributions, commute with isometries and provide
sheaf embeddings into algebras of generalized functions on $M$.\medskip\\

\noindent {\em 2010 Mathematics Subject Classification:} 58J37, 46F30, 46T30, 35A27, 35L05
{\em Keywords:} regularization of distributions, wave equation, microlocal regularity
\end{abstract}

\maketitle

\section{Introduction}

In this paper we introduce a  method of regularizing distributions on a smooth manifold by nets of  
smooth functions such that the approximating nets themselves  retain a maximal amount of information about the distribution. 
In particular, the analytical properties of interest include  the support of the distribution, its microlocal singularities 
(wavefront set), its Sobolev regularity  and its behavior (pull-back)  under certain diffeomorphisms. 
We shall first  abstractly describe what properties such an approximation  should have and then construct 
such approximations by  a suitable  choice of smoothing process. The latter is obtained using functional calculus
for the solution operator of the wave equation for the metric Laplacian (we note that the set of analytical properties
to be preserved excludes, e.g., smoothing via the heat kernel as a possible approximation procedure). 

We will give a precise formulation of the requirements to be imposed on our smoothing processes. 
In order to achieve this we need a conceptual framework that allows to assign 
geometrical and analytical properties like those mentioned above to regularizations, that is, to nets of smooth functions.
Such a framework is in fact available in the theory of 
algebras of generalized functions (\cite{C1,C2,O,NPS,GKOS}), which therefore will provide the 
underlying language for our approach. The basic idea in this theory is to express analytical
properties of distributions as asymptotic estimates in terms of a regularization parameter $\eps$.
Up to now, there is a certain dichotomy in the theory of algebras of generalized functions. On the one hand,
so-called full Colombeau algebras allow a canonical embedding of the space of Schwartz distributions
on differentiable manifolds (\cite{GFKS,GKSV}), but their elements do not depend on a single
real regularization parameter $\eps$. Instead, such generalized functions are smooth
maps on certain spaces of test functions and require a rather involved asymptotic. So-called
special Colombeau algebras, on the other hand, are modelled directly as quotients of certain
powers of the space of smooth functions, hence allow for a more straightforward modelling
of singularities. A rich geometric (e.g., \cite{DD,KS,conn,KSV,V}) and analytic (e.g., 
\cite{O,NPS,DPS,GH,HdH,HOP,G}) theory is available for 
such algebras. The drawback here is that there is no canonical
embedding of distributions into such algebras. Up to recently, only `non-geometric'
embeddings, based e.g.\ on de Rham regularizations (basically through convolution
with a mollifier in charts) were available, cf.\ \cite{DD,GKOS}. In \cite{D}, however, 
a new approach to embedding
distributions into special Colombeau algebras was put forward, namely a geometric
embedding of distributions on compact manifolds without boundary based on 
functional calculus of the Laplacian. In the present paper we follow this general
philosophy to produce geometrical embeddings for general complete Riemannian 
manifolds. A main new ingredient here is that we employ the solution operator 
for a certain initial value problem of the wave equation for our regularization
processes. We obtain a set of optimal properties for such embeddings. In particular,
they commute with isometries, respect the functional calculus of the Laplacian, 
and preserve the microlocal structure of distributions.

The paper is organized as follows: in the remainder of this introduction we fix
some notations and terminology. Section \ref{wavesec} collects a number of
results on wave equations on complete Riemannian manifolds. These preparations
are then used in section \ref{embsec} to construct optimal regularization
processes and use these to obtain geometrical embeddings of Schwartz distributions
into special Colombeau algebras. Finally, section \ref{vectorsec} shows how to extend
our approach in various directions. On the one hand, we demonstrate how to
adapt the construction to obtain embeddings for distributional sections of 
vector bundles. On the other hand, we specify the main properties of the Laplacian
that were used to obtain optimal regularization processes in section \ref{embsec} and
show that a wide class of differential operators allows to obtain analogous
regularization processes.

Throughout this paper $M$ will denote an orientable complete Riemannian manifold of dimension
$n$ with Riemannian metric $g$. The space $\D'(M)$ of Schwartz distributions on $M$
is defined as the dual of the space $\Omega_c^n(M)$ of compactly supported $n$-forms
on $M$. We write $\D(M)$ for the space of smooth compactly supported functions
on $M$. Since $M$ is orientable and Riemannian, we may identify $\D(M)$ with
$\Omega_c^n(M)$ via $f\mapsto f\cdot dg$, with $dg$ 
the Riemannian volume form induced by $g$. In this sense, $\D'(M)$ is in
fact the dual space of $\D(M)$. We consider $L^1_{\mathrm{loc}}(M)$ (hence in particular 
$\cinfty(M)$) a subspace of $\D'(M)$ via $f\mapsto [\vphi \mapsto \int_M f\vphi dg]$.
If $E$ is a vector bundle over $M$ then $\D'(M:E)$, the space
of $E$-valued distributions on $M$ is given by $\D'(M:E) = \D'(M)\otimes_{\cinfty(M)} 
\Gamma^\infty(M:E)$, with $\Gamma^\infty(M:E)$ the space of smooth sections
of $E$ (cf., e.g., \cite{GKOS}, 3.1 for details). The wavefront set of a distribution
$w\in \D'(M)$ is denoted by $\WF(w)$.

We now turn to notations from the theory of algebras of generalized functions,
where we basically adopt the terminology from \cite{GKOS, Gtop}.
Given $E$ a locally convex (Hausdorff) topological vector space, one can associate to $E$
a space $\maG_E$ of generalized functions as follows.
Let $I$ be the interval $(0,1]$. Define the smooth moderate nets in $E$ to be smooth maps 
(in the sense of \cite{KM})
$$
I\rightarrow E\quad\ep\mapsto u_{\ep}
$$ 
such that for all continuous semi-norms $\rho$ on $E$ there exists  a (negative) integer $N$
such that
\begin{eqnarray}\label{asympt}
|\rho(u_{\epsilon})| =  O(\epsilon^N)  
\quad \textrm{ as }\, \,\epsilon \rightarrow
 0.
\end{eqnarray}
Here as usual by $f(\ep) = O(g(\ep))$ as $\ep\rightarrow 0$ we mean there exists an  $\ep_0>0$ and a constant $C>0$ such that  $f(\ep)<Cg(\ep)$  for $\ep<\ep_0$.
We denote the set of all  moderate smooth nets in $E$ by  ${\mathcal M}_E$.
 Similarly we can define the negligible nets to be the smooth maps $u_{\ep}$ such that \eqref{asympt} holds
for all continuous seminorms $\rho$ on $E$ and all $N$. We  shall denote the set  of all smooth negligible nets by 
${\mathcal N}_E$.

 The space of generalized functions based on $E$ is then defined to be the quotient,
\[\maG_E:={\mathcal M}_E/{\mathcal N}_E.\]
If $E$ is a locally convex algebra then $\gs_E$ is an algebra as well.
One notes that in defining ${\mathcal M}_E$ and ${\mathcal N}_E$  it suffices to restrict to any family of
seminorms that generate the locally convex topology on $E$. If  $(u_{\ep})$ is a moderate net in ${\mathcal M}_E$ 
then  the element which it represents in the  quotient $\maG_E$  will be written as  $[(u_{\ep})]$.

When $E=\smooth(M)$ is the algebra of smooth functions on a manifold $M$ then we
write ${\mathcal M}_{\cinfty(M)}=\esm(M)$, $\ns_{\cinfty(M)}=\ns(M)$, and 
$\maG(M):=\maG_{\smooth(M)}$. $\gs(M)$ is the standard (special) Colombeau algebra
of generalized functions on $M$ (\cite{C1,DD,GKOS}).
For $E=\CC$ the space $\maG_{\CC}$ inherits a  ring structure from $\CC$ and we call it  the space of generalized numbers and 
denote it by $\tilde{\CC}$. Every space $\maG_E$ is naturally a $\tilde{\CC}$-module, and hence  is often referred to as 
the $\tilde{\CC}$-module associated with $E$ (\cite{Gtop}).  

We recall the functoriality of the above construction.
If $\phi:E\rightarrow F$ is a continuous linear map between locally
convex spaces $E $ and $F$ then there is a natural induced map $\phi_*:\maG_E\rightarrow
\maG_F$ defined on the representatives as
$\phi_*([(u_{\epsilon})])=[(\phi(u_\epsilon))]$. 
For example any smooth map between two manifolds $f:M\rightarrow N$ gives  rise
to a pullback map $f^*:\maG(N)\rightarrow \maG(M)$. As  a consequence we can define a presheaf of algebras on $M$ by  assigning to any open set 
$U\subseteq M$ the space $\maG(U)$. The restriction maps are given by the pull back under inclusions, that is if
$i:U\hookrightarrow V$ is an inclusion of open sets then $i^*:\maG(V)\rightarrow \maG(U)$ is the restriction map. This presheaf is  in fact a fine sheaf. 
Thus in particular we can define the support of a global section $u\in\maG(M)$  as usual to be the complement of the  
biggest open subset of $M$  on which   $u$ restricts to $0$.  In a similar fashion if $E\rightarrow  M$ is a  (complex) 
vector bundle then we obtain a sheaf of $\tilde{\CC}$-modules  defined as 
$\maG(M:E):=\maG_{\Gamma^{\infty}(M:E)} \cong \gs(M)\otimes_{\cinfty(M)} \Gamma^\infty(M:E)$.

For any locally convex space $E$ we can also define a subspace $\maG^{\infty}_E$ of regular elements of $\maG_E$. 
These are all elements in $\maG_E$  such that there exists an integer $N$ so that  \eqref{asympt} holds 
independently of the   
seminorm $\rho$ chosen. Again we shall denote  by $\maG^{\infty}(M)$ the algebra $\maG^{\infty}_{\smooth(M)}$.  
The algebra $\maG^{\infty}(M)$ provides the regularity features for the analysis of generalized functions in 
$\maG(M)$ in the same way that $\smooth(M)$ provides these features in $\maD'(M)$ (\cite{O,H,DPS,GH}). For instance:
\begin{enumerate}[(a)]
\item  {\it Singular support}: For $u\in\maG(M)$ the singular support is defined as the  complement of the largest open  
set $U$ on which  the restriction  $u|_U$ is in $\maG^{\infty}(U)$.
\item {\it Wavefront set}: Let $\Omega\subseteq \R^n$ be open. A generalized function $u\in \gs(\Omega)$ is
called $\gs^\infty$-microlocally regular at $(x_0,\xi_0)\in T^*\Omega\setminus 0$ if there exists some $\varphi
\in \D(\Omega)$ with $\varphi(x_0)=1$ and a conic neighborhood $\Gamma\subseteq \R^n\setminus 0$ of $\xi_0$ such that
the Fourier transform ${\mathcal F}(\varphi u)$ is rapidly decreasing in $\Gamma$, i.e., there exists $N$ such 
that for all $l$,
\begin{equation}\label{wfest}
 \sup_{\xi\in \Gamma} (1+|\xi|)^l|(\varphi u_\eps)^\wedge(\xi)| = O(\eps^{-N}) \qquad (\eps\to 0).
\end{equation}
The generalized wave front set of $u$, $\WFg(u)$, is the complement of the set of points $(x_0,\xi_0)$
where $u$ is $\gs^\infty$-microlocally regular. It follows from \cite{Ha} that, based on this
definition, for any $u\in \maG(M)$, $\WFg(u)$ can naturally be viewed as a subset of $T^*M\setminus0$. 

An alternative description of $\WFg(u)$ is as follows (\cite{GH}): Let $P$ be an order  $0$ classical pseudodifferential operator  
and let $\operatorname{char}(P)\subseteq T^*M$ be the characteristic set of $P$, that is the $0$-set in $T^*M\setminus 0$ 
of its principal symbol. Then for $u\in \gs(M)$,
\[\WFg(u)=\bigcap_{Pu\in\maG^{\infty}(M)} \operatorname{char}(P)\quad P\in\Psi^0_{cl}(M).\]
\item {\it Hypoellipticity}: An operator $P$ is said to be $\maG^{\infty}$-hypoelliptic if
for every $U\subseteq M$ open and every $u\in \gs(U)$,
\[Pu\in\maG^{\infty}(U)\Longrightarrow u\in\maG^{\infty}(U).\]
\end{enumerate}
General references for microlocal analysis in algebras of generalized functions are \cite{NPS,DPS,H,HdH,G,GH}. 

\section{The wave equation on a complete Riemannian manifold}\label{wavesec}
In our approach, optimal regularization processes on complete Riemannian manifolds will be based on the solution operator
for the wave equation. To allow for a smooth presentation, in the present section we therefore collect some basic properties of 
solutions of the wave equation in this global setting.
 
Let $(M,g)$  be an oriented, connected complete Riemannian manifold (without  boundary) of dimension $n$ and denote by $\Delta$ the Laplace operator on $M$.
The Riemannian metric $g$ induces a volume form $dg$ on $M$, and
we will denote the corresponding $L^2$-norm by $\|\,\|_{L^2(M)}$.
On differential forms, the corresponding inner product is given by $(\alpha,\beta) := \int \alpha * \beta \equiv \int \alpha\wedge *\beta$.

 Let $d$ be the exterior differential on the space $\Omega^*(M)$ of differential
forms on $M$ and denote by $*$ the Hodge star operator. Then the codifferential $\delta$ on $\Omega^*(M)$
is defined, for any $k$-form $\alpha$, by $\delta \alpha = (-1)^{nk+n+1} * d *\alpha$. Finally, the
Laplace operator on $\Omega^*(M)$ is defined by
$\Delta:= (d+\delta)^2=d\circ \delta + \delta\circ d$.
This sign convention renders $\Delta$ a positive operator  on $L^2(M)$ (cf.\ \cite{gaffney}),
and  for any smooth function $u$   in particular $\Delta u = - \mathrm{div} \mathrm{grad} u$.

The operators $d$, $\delta$ and $\Delta$ are unbounded on the Hilbert space $L^2(M:\Lambda^*M)$.
The natural domain of $d$ is given by
$\Dom(d):=\{\alpha\in \Omega^*(M)\mid \|\alpha\|\,, \|d\alpha\| < \infty\}$, and analogously for $\delta$. 
This fixes the natural domain of $\Delta$ to be 
$$
\Dom(\Delta):=\{\alpha\in \Dom(d)\mid d\alpha \in \Dom(\delta)\} \cap \{\alpha\in \Dom(\delta)\mid \delta\alpha \in \Dom(d)\}. 
$$
We will mainly be interested in the restriction of $\Delta$ to $L^2(M,dg)$, which is an 
unbounded essentially self-adjoint operator with dense domain (cf.\ \cite{gaffney}).

We consider the following initial value problem for the wave equation on $M$ (or, strictly speaking, on $\R \times M$):
\begin{eqnarray}\label{wave_eqn}
(\frac{\partial^2}{\partial s^2}+\Delta)u=0 \label{wave}\\
u(0,x)=u_0(x) \label{waveic1}\quad
\frac{\partial}{\partial s}u(0,x)=0 \label{waveic2}
\end{eqnarray}
Since $g$ is complete, the Laplace operator is self-adjoint and the above wave equation 
has a  unique (mild, hence distributional) solution in ${\mathcal C}(\R,L^2(M))$ 
for all $u_0$ in  $L^2(M)$. By functional calculus this solution 
can be written as  $\operatorname{cos}(s\sqrt{\Delta})u_0$.
\begin{remark} We briefly sketch a  proof to the existence and uniqueness result for the above wave equation. 
Since  $\Delta$ is a positive self-adjoint operator, we may equivalently rewrite (\ref{wave})--(\ref{waveic2}) as a first order
initial value problem on the Hilbert space $H:= L^2(M)\oplus L^2(M)$:
\begin{eqnarray*}
\frac{d}{ds}
\begin{pmatrix}
u \\
v
\end{pmatrix}
&=&
\begin{pmatrix}
-i\sqrt{\Delta}  & I \\
0	& i\sqrt{\Delta}
\end{pmatrix}
\!\!
\begin{pmatrix}
u \\
v
\end{pmatrix}
\\
u(0,\,.\,) = u_0 && v(0,\,.\,) = v_0 := i\sqrt{\Delta}u_0
\end{eqnarray*}
We set
$$
A_0:= \begin{pmatrix}
0  & I \\
0	& 0
\end{pmatrix}
\quad A_1:= \begin{pmatrix}
-\sqrt{\Delta}  & 0 \\
0	& \sqrt{\Delta}
\end{pmatrix}
\quad
w_0 := 
\begin{pmatrix}
u_0\\
v_0
\end{pmatrix}
$$
By the theory of unitary semi groups (e.g., \cite{pazy}), $iA_1$ generates a strongly continuous unitary
group $U(s) = \exp(isA_1)$. Since $A_0$ is bounded, $A:=A_0+iA_1$ also generates a strongly continuous
semigroup. Consequently, the above initial value problem with $w_0\in \Dom(A)$ is uniquely solvable.
More explicitly, for $w_0$ in the dense subspace $D^\infty := \bigcap_{k=0}^\infty \Dom(A^k)$, the power series
expansion of $\exp(sA)w_0$ readily shows that $u(s) = \cos(s\sqrt{\Delta})u_0$ on a dense subspace,
hence in fact for all $u_0 \in L^2(M)$ (cf.\ also \cite{CGT, taylor}).
\end{remark}
For a given even Schwartz function $F\in\Sch(\RR)$ the operator $F(\sqrt{\Delta})$  can be defined by 
functional calculus for essentially self adjoint  unbounded operators. 
We show that we have the following alternative description of $F(\sqrt{\Delta})$ by inverse Fourier transform:
\begin{equation}\label{fourier}
F(\sqrt{\Delta})=\frac{1}{2\pi}\int_{-\infty}^{\infty} \hat{F}(s)\cos(s\sqrt{\Delta})\,ds.
\end{equation}
In fact, by standard estimates in the functional calculus this identity holds pointwise
on $L^2(M)$. Moreover, since
$$
\| \hat F(s)\cos(s\sqrt{\Delta})\| \le |\hat F(s)| \|\cos(s\,.\,)\|_{L^\infty} \le |\hat F(s)|
$$
and $\hat F \in L^1(\R)$, the integral in (\ref{fourier}) exists as a Bochner integral in $\mathcal{B}(L^2(M))$.

The description of functional calculus using (\ref{fourier})  allows to estimate the support of 
the kernel of the operator $F(\sqrt{\Delta})$  based on the finite speed of propagation for the 
operator $\cos(s\sqrt{\Delta})$. To see this, we first describe an alternative approach
to obtaining the solution operator to (\ref{wave})--(\ref{waveic2}). Consider the first order differential
operator $D:= d + \delta$ on the space $\Omega^*(M)$ of differential forms on $M$. 
The symbol $\sigma_D$ of $D$ is given by $\sigma_D(x,\xi) = \xi\wedge \,.\, -  i_\xi$, 
where $i_\xi$ denotes interior differentiation
along the vector field metrically equivalent to the one-form $\xi$ (cf., e.g., \cite{N}, 10.1.22).
Therefore, 
$$
\sigma_D(x,\xi)^2 = (\xi\wedge \,.\, -  i_\xi)^2 = -\|\xi\|^2 \,\mathrm{id}
$$
(again by \cite{N}, 10.1.22).
We conclude that the speed of propagation of $D$, defined by 
$c_D:= \sup \{\|\sigma_D(x,\xi)\|\mid x\in M, \xi\in T_xM^*, \|\xi\|=1\}$ is $c_D=1$.
Since $D$ is symmetric and of finite propagation speed, it is essentially self-adjoint (\cite{HR}, 10.2.11). 

Based on these facts, an explicit bound on the speed of propagation for the support of $e^{isD}u$ is given 
by the following result (see \cite{HR}, 10.5.4):

\begin{prop} \label{propsupp}
Let $u\in L^2(M:\Lambda^*M)$ and denote by $d_g$ the distance function induced by $g$. 
Then $\supp(e^{isD}u) \subseteq B_{|s|}(\supp(u)):=\{x\in M \mid d_g(x,\supp(u)) \le |s|\}$.
\end{prop}

Thus the same property holds for the bounded operator 
$$
\cos(sD) = \frac{1}{2}(e^{isD} + e^{-isD})
$$
By functional calculus, for any $u\in L^2(M)$, $\cos(sD)u$ solves the initial value problem
(\ref{wave_eqn})--(\ref{waveic2}). By uniqueness, therefore, $\cos(sD) = \cos(s\sqrt{\Delta})$
on $L^2(M)$ and the above considerations apply to our solution operator.

Next we provide some estimates that will repeatedly be useful in our further
study of operators of the form $F(\sqrt{\Delta})$.
\begin{lem}\label{norm_estimates}
Let $F\in {\mathcal S}(\R)$ be even. Then for any compactly supported smooth function $u$,  
\[\|F(\sqrt{\Delta})u\|_{L^2(M)}\leq\|u\|_{L^2(M)}\frac{1}{\pi}\int_0^{\infty}|\hat{F}(s)|\,ds.\]
Moreover, for any positive integers $k,l$,
\[\|\Delta^kF(\sqrt{\Delta})\Delta^lu\|_{L^2(M)}\leq\|u\|_{L^2(M)}\frac{1}{\pi}
\int_0^{\infty}|{\hat F}^{(2k+2l)}(s)|\,ds.\]
\end{lem}
\begin{proof}
The first estimate follows from \eqref{fourier}:  
Since the operator $\cos(s\sqrt{\Delta})$  has operator norm $\le 1$ we have
\begin{align*}
\|F(\sqrt{\Delta})u\|&=\|\frac{1}{2\pi}\int_{-\infty}^{\infty}\hat F(s)\cos(s\sqrt{\Delta})u\,ds\|\\
&\leq\|u\|\frac{1}{\pi}\int_0^{\infty}|\hat{F}(s)|ds.
\end{align*}
Concerning the second inequality, note that by functional calculus we have
\begin{eqnarray*}
\Delta^kF(\sqrt{\Delta})\Delta^l u &=& (t^{2k+2l}F(t)) (\sqrt{\Delta}) u \\
&=& \frac{1}{2\pi} \int_{-\infty}^\infty  (t^{2k+2l}F(t))^{\wedge} \cos(s\sqrt{\Delta})u\,ds\\
&=& \frac{1}{2\pi} (-1)^{k+l} \int_{-\infty}^\infty {\hat F}^{(2k+2l)}(s) \cos(s\sqrt{\Delta})u\,ds,
\end{eqnarray*}
from which the claim follows as above.
\end{proof}

For any $s\in \R$ and any $u\in \D(M)$ we set
$$\|u\|_s := \| (1+\Delta)^{s/2} u \|_{L^2(M)}$$
The Sobolev space of order $s$ is the completion of $\D(M)$ with respect to this norm. 
We set $H^\infty(M) := \bigcap_{s\in \R} H^s(M)$ and denote by $\hscomp(M)$ the  space of compactly supported 
elements of $H^s(M)$.

The following result will be essential for our approach to regularizing distributions
on complete Riemannian manifolds. For the notion of (properly supported) smoothing (or regularizing) operator we
refer to \cite{CP}, ch.\ 1.4. 

\begin{prop}\label{Gsmooth} Let $F\in {\mathcal S}(\R)$ be even. Then
\begin{itemize}
\item[(i)]  The operator $F(\sqrt{\Delta}): {\mathcal D}'(M) \to \cinfty(M)$ is a smoothing operator.
\item[(ii)] Let $c>0$ and let $\phi_c\in \D(\R)$ be such that  $\supp(\phi_c) \subseteq [-2c,2c]$
and $\phi_c \equiv 1$ on $(-c,c)$. Then
$$
T(\sqrt{\Delta}) := \frac{1}{2\pi} \int_{-\infty}^\infty \phi_c(s) \hat F(s) \cos(s\sqrt{\Delta})\, ds,
$$
is a properly supported smoothing operator.
\end{itemize}
\end{prop}
\begin{proof}
(i) This is a consequence of  ellipticity of the  Laplace operator  $\Delta$. For brevity, 
we set $A := F(\sqrt{\Delta})$. Given any $\varphi\in \D(M\times M)$, let
$$
\langle K,\vphi\rangle := \int_M (A\vphi(x,\,.\,))(x)\,dx.
$$
Then for all $\psi_1$, $\psi_2 \in \D(M)$,  
$$
\langle K, \psi_1\otimes \psi_2\rangle = \int_M \psi_1(x) A\psi_2(x) \,dx = \langle A\psi_2, \psi_1\rangle,
$$
so $K$ is the distributional kernel of $A: \D(M) \to \cinfty(M) \subseteq \D'(M)$. We have to show that
$K\in \cinfty(M\times M)$. By Lemma \ref{norm_estimates}, for all $k\in \N_0$ and all $\psi\in \D(M)$,
$\|A\psi\|_k \le C_k \|\psi\|_{L^2(M)}$, hence
$$
|\langle K,\psi_1\otimes \psi_2\rangle| \le \int_M |\psi_1(x)| |A\psi_2(x)| \,dx \le C_0 \|\psi_1\|_{L^2(M)}
\|\psi_2\|_{L^2(M)}
$$
Since $\D(M)\otimes \D(M)$ is dense in $L^2(M\times M)$ this implies that $K\in L^2(M\times M)$. Now set
$A_{k,l} := \Delta^l A \Delta^k$. Then the kernel of $A_{k,l}$ is given by $K_{k,l}:= \Delta^k_y \Delta^l_x K$,
since
$$
\langle K_{k,l}, \psi_1\otimes \psi_2\rangle = \langle K, \Delta^l \psi_1\otimes \Delta^k\psi_2 \rangle
= \langle A \Delta^k\psi_2, \Delta^l \psi_1\rangle = \langle A_{k,l} \psi_2,\psi_1\rangle .
$$
As above it follows that $K_{k,l} \in L^2(M\times M)$ for all $k$, $l$, hence 
by elliptic regularity for $\Delta_x\otimes 1+1\otimes \Delta_y$  it follows that $K$ is  smooth.

(ii) $T(\sqrt{\Delta})$ is a smoothing operator by (i). To establish proper support, by \cite{F}, Prop.\ 8.12 we have to show:
\begin{itemize}
\item[(a)] $\forall K\comp M$ $\exists L\Subset M$ such that $u\in \D(K) \Rightarrow T(\sqrt{\Delta})u \in \D(L)$.	
\item[(b)] $\forall K\comp M$ $\exists L\Subset M$ such that $u=0$ on $L$ $\Rightarrow T(\sqrt{\Delta})u = 0$
on $K$.
\end{itemize}
Both (a) and (b) follow from the finite speed of propagation of $\cos(\sqrt{\Delta})$ which implies that  
there exists some $\tilde C>0$ such that for any $u\in L^2(M)$,
the support of $T(\sqrt{\Delta})u$ is contained in a ball of radius $\tilde C$ around $\supp(u)$ (Prop.\ \ref{propsupp}). 
The result therefore follows from the properness of the complete metric $g$.
\end{proof}

\section{Embeddings}\label{embsec}
In this section we will employ the smoothing operators developed in Section \ref{wavesec} to construct optimal
embeddings of the space $\D'(M)$ of distributions on a complete Riemannian manifold $M$ into the algebra $\gs(M)$ of
generalized functions on $M$.

A set  $X\subset M\times M$ is called proper if the restriction of  the projections on both factors  $\pi_j:X\rightarrow M: j=1,2$  are proper maps.
Let $\Psi^{-\infty}_{\mathrm{prop}}(M)$ be the space of all operators $T:\smooth(M)\rightarrow\D(M)$ with smooth kernels with proper support in $M\times M$.
By a regularization process we mean a  net $T_{\varepsilon}$  of properly  supported smoothing operators which 
provides an approximate identity on  compactly supported distributions. 
More precisely, we shall be interested in  rapidly converging regularization processes of the following kind:
 
 \begin{definition}\label{abstract_embedding}
  A parametrized family $(T_{\varepsilon})_{\eps\in I}$ of properly supported smoothing operators is
 called an {\em optimal regularization process} if
 \begin{enumerate}[(A)]
 \item \label{moderate} The regularization of any compactly supported distribution is of moderate growth. That is, for any continuous semi-norm 
$\rho$ on $\smooth(M)$ and any distribution $w\in \maE'(M)$ there
 exists an integer $N$ such that
\[
\rho(T_{\varepsilon}w) = O(\varepsilon^{N}) \qquad (\eps\to 0),
\]
i.e., $(T_\eps w) \in \esm(M)$.
\item\label{Identity} The net $(T_{\varepsilon})$ is an approximate identity: for each compactly supported distribution 
$w\in\maE'(M)$
\[\underset{\varepsilon \rightarrow 0}{\operatorname{lim}}T_{\varepsilon}w=w\quad \textrm{in}\,\,\maD'(M).\]

\item \label{local} 
Preservation of supports: For any $w\in \maE'(M)$, $\supp(w)$ equals the $\gs(M)$-support $\supp [(T_\eps w)]$
of the class $[(T_\eps w)]$.

\item\label{negligible} If $u\in\D(M)$ is a smooth compactly supported function  on M then for all 
continuous semi-norms $\rho$ on $\smooth(M)$ 
and given any integer $m$,
\[
\rho(T_{\varepsilon}u-u) = O(\varepsilon^m),
\]
i.e., $(T_\eps u - u) \in \ns(M)$. 

\item\label{singularity} Preservation of wavefront sets:  Setting $\iota_T: w\mapsto [(T_{\varepsilon}w)]$,
for any $w\in \maE'(M)$ we have 
$$
\WF(w) = \WF_g(\iota_T(w)).
$$ 

\end{enumerate}
\end{definition}

\noindent Given an optimal regularization process $(T_\eps)$, we obtain a linear embedding 
\begin{eqnarray*}
\iota_T: \maE'(M) &\to& \gs(M) \\
\iota_T(w) &=& [(T_\eps w)]
\end{eqnarray*}
(by (\ref{moderate}) and (\ref{Identity})). By (\ref{local}), $\iota_T$ extends to an embedding of $\D'(M)$ into $\gs(M)$ 
which preserves supports. More precisely, there is a unique sheaf morphism on $\D'(M)$ (also denoted by $\iota_T$) which extends
$\iota_T: \maE'(M) \to \gs(M)$. $\iota_T$ is a linear embedding that commutes with restrictions. Details on how to 
extend $\iota_T$ from $\maE'(M)$ to $\D'(M)$ can be found in \cite{DD}, Sec.\ 2 or in \cite{GKOS}, 1.2 (although carried
out for special cases of optimal embeddings in these references, the arguments given there, entirely sheaf-theoretic in nature, 
carry over to the general situation studied here). 
(\ref{negligible}) implies that $\iota_T$ renders $\cinfty(M)$ a faithful subalgebra of $\gs(M)$. 
Finally,
(\ref{singularity}) secures preservation of wavefront sets and, therefore, of singular supports under this embedding.
In particular, precisely the distributions that map into 
the subalgebra $\maG^{\infty}(M)$ under $\iota_T$ are smooth:
\[\iota_T(\maD'(M))\cap \maG^{\infty}(M)= \iota_T(\smooth(M)).\]

First, we provide some examples of optimal regularization processes:

\begin{eg}\label{funct_calc}
 Let $\Delta$ be  the Laplace operator associated to  a closed Riemannian manifold $M$. Let $F\in\Sch(\RR)$ be a Schwartz function 
on the reals such that $F$ is identically $1$ near the origin. Let $F_{\ep}(x):=F(\ep x)$. Then by applying standard functional 
calculus, $F_{\ep}(\Delta)$ is an optimal regularizing process: For a closed manifold, Weyl's estimates on the
spectrum of the Laplacian provide asymptotic bounds for the spectral counting function 
\[N_{\Delta}(\lambda)=\#\{\lambda_k|\,\lambda_k<\lambda\}.\] 
In fact, for $m=\operatorname{dim}(M)$ we have
\[N_{\Delta}(\lambda) \sim \frac{\operatorname{vol}(M)} {(4\pi)^{\frac{m}{2}}\Gamma(m/2+1)}\lambda^{\frac{m}{2}}.\]
Essentially, this suffices to obtain all the estimates in Definition \ref{abstract_embedding} (preservation
of wavefront sets follows as in Th.\ \ref{wfth} below). In addition, $F_{\ep}(\Delta)$  is invariant under isometries. We refer to \cite{D} for details. 
\end{eg}
\begin{eg}
 As in the original construction of Colombeau (cf., e.g., \cite{C1,C2,GKOS}) an optimal regularization process can be 
constructed from a mollifier $\rho\in\Sch(\RR^n)$ satisfying the following  conditions:
\begin{eqnarray}  \label{nomoment}
\int_{\RR^n}\rho(x)dx=1\quad\int_{\RR^n}x^{\alpha}\rho(x)dx=0\hskip 0.2in \alpha\in\NN_+^n.
\end{eqnarray}
Then the net of functions  $\rho_{\ep}(x):=\frac{1}{\ep^n}\rho(\frac{x}{\ep})$ is a delta net. 
Convolution with such a delta net provides an example of an optimal regularization process. 
For example, estimate (\ref{negligible}) from Definition \ref{abstract_embedding}  can be established in this setting
using Taylor's theorem and the moment conditions \eqref{nomoment} imposed on $\rho$.  Concerning 
\eqref{singularity} from Def.\ \ref{abstract_embedding},
see \cite{NPS,H}. An important characteristic of these approximate units is their equivariance with respect to 
the Euclidean translations.
\end{eg}

For $(M,g)$ a complete Riemannian manifold with Laplacian $\Delta$, $F\in \maS(\R)$ even, and $\phi_c$ as in
Prop.\ \ref{Gsmooth}, we additionally suppose that $F$ equals $1$ in a neighborhood of $0$ and that 
$\phi_c$ is even. For any $\eps\in I$ we set $F_\eps(s):=F(\eps s)$, and 
\begin{equation} \label{tepsdef}
\teps := \frac{1}{2\pi} \int_{-\infty}^\infty \phi_c(s) (F_\eps)^\wedge(s) \cos(s\sqrt{\Delta})\, ds.
\end{equation}
We will show that the family of smoothing operators $(\teps)_{\eps\in I}$ is an optimal regularization
process in the sense of Def.\ \ref{abstract_embedding}. Each $\teps$ is a
properly supported smoothing operator by Prop.\ \ref{Gsmooth} (ii). Furthermore, we note that the explicit form
of (\ref{tepsdef}) allows to view it as a generalized integral operator in the sense of \cite{BCD,Del}.

Turning first to Def.\ \ref{abstract_embedding} (\ref{moderate}), we have:
\begin{prop} \label{AProp}
Let $u\in \maE'(M)$. Then $(\teps u) \in \esm(M)$.
\end{prop}
\begin{proof}
It follows immediately from (\ref{tepsdef}) that $(\eps,x)\mapsto (\teps w)(x)$ is smooth,
so it remains to establish the moderateness estimates for the net $(\teps w)$.
Since $w\in \maE'(M)$ there exists some $s_0\in \R$ with $w\in \hsncomp(M)$. By Prop.\ \ref{propsupp}, there exists some
fixed compact set $K\comp M$ such that $\supp(\teps w) \subseteq K$ for all $\eps\in I$. Moreover,
Prop.\ \ref{Gsmooth} implies that each $\teps w$ is in $H^\infty_{\mathrm{cp}}(M)$. 
By the local Sobolev embedding theorem it therefore suffices to show that for each $s\in \R$ there exists 
some $N\in \N$ such that 
$$
\|\teps w\|_s = O(\eps^{-N}).
$$
In fact (by enlarging $s$ if necessary) we may assume in addition that $l:=\frac{s-s_0}{2}\in \N$. 
Let $u$ be the unique element of $L^2(M)$ such that $w=(1+\Delta)^{-s_0/2}u$. Then
\begin{eqnarray*}
 \|\teps w\|_s &=& \|(1+\Delta)^{s/2}\teps(1+\Delta)^{-s_0/2} u)\|_{L^2(M)} \\
&=& \|(1+\Delta)^l\teps u\|_{L^2(M)}\\ 
&\le& \sum_{j=0}^l 
\binom{l}{j}
\|\Delta^j\teps u\|_{L^2(M)}.
\end{eqnarray*}
Now write $\phi_c = (\psi_c)^\wedge$ for some $\psi_c \in \maS(\R)$. Then
$$
\teps u = \int_{-2c}^{2c} (\psi_c * F_\eps)^\wedge(t) \cos(t\sqrt{\Delta})u\,dt .
$$
Since $\psi_c * F_\eps$ is even, Lemma \ref{norm_estimates} implies that
$$
\|\Delta^j \teps u\|_{L^2(M)} \le \|u\|_{L^2(M)} \frac{1}{\pi}\int_{0}^\infty |[(\psi_c * F_\eps)^\wedge]^{(2j)}(t)|\,dt
$$
From this, we finally obtain
\begin{eqnarray}\label{Aproplast}
\|\Delta^j \teps u\|_{L^2(M)} &\le& \|u\|_{L^2(M)} \frac{1}{\pi}\int_{0}^{2c} |[(\phi_c(t) \frac{1}{\eps} \hat F(\frac{t}{\eps}))]^{(2j)}(t)|\,dt\\
&=& O(\eps^{-2l-1}) \nonumber
\end{eqnarray}
for $0\le j\le l$.
\end{proof}

We may use the method of proof of Prop.\ \ref{AProp} to show that the embedding $\iota_T$ is
in fact independent of the particular choice of the cut-off function $\phi_c$:
\begin{lem} \label{philemma}
Suppose that $\phi^1_{c_1}$ and $\phi^2_{c_2}$ are cut-off functions as in (\ref{tepsdef}) and denote the 
corresponding regularization operators by $T_\eps^{1}(\sqrt{\Delta})$ and $T_\eps^{2}(\sqrt{\Delta})$,
respectively. Then for any $w\in \maE'(M)$, $\iota_{T^1}(w) = \iota_{T^2}(w)$.
\end{lem}
\begin{proof}
Set $\tilde \phi := \phi^1_{c_1} - \phi^2_{c_2}$. Using the assumptions and notations from 
the proof of Prop.\ \ref{AProp}, we have to estimate 
$$
\lt{\Delta^j (T_\eps^{1}(\sqrt{\Delta})u - T_\eps^{2}(\sqrt{\Delta})u)} .
$$
Since 
$$
T_\eps^{1}(\sqrt{\Delta})u - T_\eps^{2}(\sqrt{\Delta})u =
\frac{1}{2\pi} \int_{-\infty}^\infty \tilde \phi(s) (F_\eps)^\wedge(s) \cos(s\sqrt{\Delta})u\, ds,
$$
analogous to (\ref{Aproplast}) the $L^2$-norm of this expression is bounded by 
$$
\|u\|_{L^2(M)} \frac{1}{\pi}\int_{0}^\infty |[(\tilde\phi(t) \frac{1}{\eps} \hat F(\frac{t}{\eps}))]^{(2j)}(t)|\,dt.
$$
A typical term to estimate therefore is 
$$ 
\int_{0}^\infty |{\tilde\phi}^{(p)}(\eps r) {\hat F}^{(q)}(r)|\,dr \quad(p,q \in \N_0)
$$
Since $F\equiv 1$ near $0$, all higher moments $\int_{-\infty}^\infty r^k \hat F(r)\,dr$ of 
$\hat F$ ($k\ge 1$) vanish.  Moreover, $\tilde \phi$ is identically zero in a neighborhood of $0$, 
so Taylor expansion of ${\tilde\phi}^{(p)}(\eps r)$ around zero up to order $m-1$ implies that
the above integral is of order $\eps^m$ for any $m\in \N$.
\end{proof}

Next, we establish suitable convergence of $\iota_T(w)$ to $w$:
\begin{prop} \label{Bprop}
Let $w\in \maE'(M)$. Then $\teps w \to w$ in $\D'(M)$.
\end{prop}
\begin{proof}
We may write $w=(1+\Delta)^k u$ for some $u\in L^2(M)$ and some $k\in \N_0$. 
By Prop.\ \ref{norm_estimates} it therefore suffices to show that $\teps u - u \to 0$ in
$L^2(M)$ in order to ensure that $\teps w \to w$ in $H^{-2k}(M)$ and hence in $\D'(M)$.
Now
\begin{eqnarray*}
&&\lt{\teps  u -  u} \\
&&\le \int_{-\infty}^\infty \lt{\phi_c(\eps r) \hat F(r)(\cos(\eps r \sqrt{\Delta})u -  u)} \, dr \to 0
\end{eqnarray*}
by dominated convergence (noting that the integrand is pointwise bounded by $2\|\phi_c\|_{L^\infty}
\|u\|_{L^2} |\hat F(r)|$).
\end{proof}

This settles Def.\ \ref{abstract_embedding}, (\ref{Identity}). As was remarked after Def.\ \ref{abstract_embedding}, we thereby
obtain a linear embedding $\iota_T$ of $\maE'(M)$ into $\gs(M)$. Our next result establishes preservation of 
supports under $\iota_T$.
\begin{prop}
For any $w\in \maE'(M)$, $\supp(w) = \supp(\iota_T(w))$.
\end{prop}
\begin{proof}
Let $x \in M\setminus \supp(w)$ and choose a compact neighborhood $K$ of $x$ such that $K\cap \supp(w) = \emptyset$.
Suppose first that $w$ is continuous. Then
$$
\iota_T(w)_\eps(x) =\int_{-\infty}^\infty \phi_c(\eps r) \hat F(r) \cos(\eps r \sqrt{\Delta})w(x)\, dr.
$$
We split this integral in one part over $|r|<2c/\sqrt{\eps}$ and a second part where $|r|>2c/\sqrt{\eps}$.
For the first part we note that by Prop.\ \ref{propsupp}, $\supp(\cos(\eps r\sqrt{\Delta})w) \subseteq
B_{\eps |r|}(\supp(w)) \subseteq B_{\sqrt{\eps} 2c}(\supp(w))$ for all $|r|<2c/\sqrt{\eps}$. Hence for small
$\eps$, the support of the first term lies in the complement of $K$.

To estimate the second integral we proceed as in the proof of Prop.\ \ref{AProp}: observe first that by Prop.
\ref{propsupp}, the support of $\cos(\eps r \sqrt{\Delta})$ is contained in a single compact set for all
$\eps$ and all $r$ in the domain of integration. It therefore suffices to estimate, for each $j\in \N_0$:
\begin{eqnarray*}
&& \lt{\int_{|r|>2c/\sqrt{\eps}}\phi_c(\eps r) \hat F(r) (\cos(\eps r \sqrt{\Delta})\Delta^j w)\, dr} \\
&& \le \|w\|_{2j} \int_{|r|>2c/\sqrt{\eps}} |\hat F(r)|\, dr = O(\eps^m)
\end{eqnarray*}
for each $m$. Summing up it follows that $x$ does not lie in the support of $\iota_T(w)$ in $\gs(M)$.
In the general case where $w$ is not necessarily continuous we can write $w = (1+\Delta)^k v$ for some continuous $v$ 
and some $k\in \N_0$, so the above argument readily carries over.

Conversely, let $x\in\supp(w)$ and suppose that there exists a neighborhood $U$ of $x$ such that
$\iota_T(w)|_U = 0$ in $\gs(M)$. Pick some $\vphi\in \D(U)$ such that $\langle w,\vphi\rangle \not=0$.
Then $|\langle \iota_T(w)_\eps,\vphi\rangle| = O(\eps^m)$ for each $m$ but $\langle \iota_T(w) - w, \vphi\rangle 
\to 0$ by Prop.\ \ref{Bprop}, so we arrive at a contradiction.
\end{proof}

\begin{prop}\label{DProp}
Let $u\in \D(M)$. Then $(\teps u - u)_{\eps\in I} \in \ns(M)$.
\end{prop}
\begin{proof}
By the local Sobolev embedding theorem, it suffices to show that for all $j\in \N_0$
$$
\alpha(\eps,j):=\lt{\Delta^j(\teps u -u)} = O(\eps^m)
$$
for each $m\in \N$. Due to our assumptions on $F$ and $\phi_c$, $\alpha(\eps,j)$ equals
\begin{eqnarray*}
&& \lt{\Delta^j \int_{-\infty}^\infty (F_\eps)^\wedge(t)(\phi_c(t) \cos(t\sqrt{\Delta})u -  u)\,dt } = \\
&& \lt{\int_{-\infty}^\infty {\hat F}(r)(\phi_c(\eps r) \cos(\eps r\sqrt{\Delta})\Delta^j u - \Delta^j u)\,dr}
\end{eqnarray*}
By Taylor expansion, for any $m\in \N$ there exists some $C_m$ such that 
$$
\phi_c(\eps r) \cos(\eps r\sqrt{\Delta})\Delta^j u = \Delta^j u + \sum_{l=1}^{m-1} \frac{\eps^l r^l}{l!} a_l \Delta^{j+l/2} u
+R_m(r,\eps) \Delta^{j+m/2} u
$$
where $a_j\in \R$ and $R_m$ is globally bounded by $C_m \eps^m$. 

Since all higher moments  of $\hat F$ vanish, 
\begin{eqnarray*}
\alpha(\eps,j) &\le& \lt{\int_{-\infty}^\infty \hat F(r) R_m(r,\eps) \Delta^{j+m/2} u \,dr} \\ 
&\le& C_m \|\hat F\|_{L^1(\R)}
\lt{\Delta^{j+m/2} u} \eps^m ,
\end{eqnarray*}
as claimed.
\end{proof}
The following important invariance properties of the embedding $\iota_T$ follow immediately
from (\ref{tepsdef}).

\begin{prop} \label{isometries} \ 
\begin{itemize} 
\item[(i)] Let $f:M\to M$ be an isometry. Then for any $u\in \D'(M)$, $\iota_T(f^*u) = f^*\iota_T(u)$.
\item[(ii)] If $\Psi$ is a pseudodifferential operator commuting with $\Delta$, then $\Psi$
commutes with $\iota_T$.
\end{itemize}
\end{prop}

Turning now to the singularity structure of distributions and their embedded images, we will show that
the embedding $\iota_T$ preserves the wavefront set of distributions, i.e., that
(\ref{singularity}) from Def.\ \ref{abstract_embedding} is satisfied. 
\begin{theo}\label{wfth}
Let $w\in \D'(M)$. Then $\WF(w) = \WFg(\iota_T(w))$.
\end{theo}
\begin{proof}
We first note that the notion of wavefront set (both distributional and generalized) is local in nature. 
Moreover, by finite propagation speed of the solution 
operator $\cos(s\sqrt{\Delta})$ (Prop.\ \ref{propsupp}), we may choose the cutoff function $\phi_c$
in such a way that given some local chart $(\psi,U)$ and any $w\in \maE'(U)$, each $\teps w$ is supported in
$U$ as well (this particular choice of $\phi_c$ does not affect $\iota_T$ by Lemma \ref{philemma}). 
By unique solvability of the wave equation (\ref{wave})--(\ref{waveic2}) we may therefore 
use (\ref{tepsdef}) on $\psi(U)$ (with metric $\psi_* g$), thereby effectively transferring
the problem to $\R^n$.

Suppose first that $(x_0,\xi_0)\in (T^*M\setminus \{0\})\setminus \WFg(w)$. By \eqref{wfest} this means there exists
some conic neighborhood $\Gamma$ of $\xi_0$ in $T^*M\setminus \{0\}$, some $N\in \N_0$ and some $\vphi \in
\D(\R^n)$ with $\varphi(x_0)=1$ such that for all $l\in \N_0$ and all $\xi\in \Gamma$,
\begin{equation} \label{nogwf}
|(\vphi(\teps w))^\wedge(\xi)|(1+|\xi|)^l = O(\eps^{-N}).
\end{equation}
For a suitable $k\in \N_0$ we may write $w=(1+\Delta)^k u$, with $u\in H^2(\R^n)$. 
Since $\Delta$ is elliptic, $\WF(w) = \WF(u)$ and $\WF_g(\iota_T(w)) = \WFg(\iota_T(u))$ (by \cite{GH}, Th.\ 4.1
and Prop.\ \ref{isometries} (ii)), so
we may without loss of generality assume that $w\in H^2(\R^n)$.
We have to estimate
\begin{eqnarray}
|(\vphi w)^\wedge(\xi)| &\le& 
|[(\teps w - w)\varphi]^\wedge(\xi)| + |(\vphi \teps w)^\wedge(\xi)| \nonumber \\ 
&\le& 
\|\vphi(\teps w - w)\|_{L^1} + |(\vphi \teps w)^\wedge(\xi)|. \label{phiw}
\end{eqnarray}
Now
\begin{eqnarray*}
&& \int |\vphi(x)(\teps w(x) - w(x))|\,dx \\
&&\le \int \int |\vphi(x)| |\phi_c(\eps r)\cos(\eps r \sqrt{\Delta})w(x) -w(x)|\,dx
|\hat F(r)|\,dr \\
&& \le C(\vphi) \int \|\phi_c(\eps r)\cos(\eps r \sqrt{\Delta})w -w\|_{L^2} |\hat F(r)|\, dr
\end{eqnarray*}
We use functional calculus to bound this term (cf., e.g., \cite{RS}, Th.\ VIII.4).
Let $U: L^2(\R^n,g) \to L^2(\Omega,d\mu)$ be a unitary
isomorphism transforming $\Delta$ into the multiplication operator $M_f: h \mapsto f h$ (for some fixed $f\in L^2(\Omega,d\mu)$).
Then setting $\alpha(\eps,r):= \|\phi_c(\eps r)\cos(\eps r \sqrt{\Delta})w -w\|_{L^2}$
and $k(\Delta):=\phi_c(\eps r)\cos(\eps r \sqrt{\Delta}) - I$ we find
\begin{eqnarray*}
\alpha(\eps,r)^2 &=& \|U(k(\Delta)w)\|_{L^2(\Omega)}^2 = \|M_{k\circ f}Uw\|_{L^2(\Omega)}^2 \\
&=& \int |k\circ f(\omega)|^2 |(Uw)(\omega)|^2 \, d\mu(\omega) \\
\end{eqnarray*}
To estimate this term we note that $k(0)=\phi_c(\eps r) - 1$ and
$$
|k'(x)| = \eps r \left|\frac{\phi_c(\eps r)}{2}\right| \left| 
\frac{\sin(\eps r \sqrt{x})}{\eps r \sqrt{x}}\right| \le C\eps r
$$
Hence $|k(x)|\le C \eps r (|x|+1)$ and we obtain
$$
\alpha(\eps,r)^2\le C \eps^2 r^2 \int (f(\omega)^2 + 1) |(Uw)(\omega)|^2 \, d\mu(\omega)
$$

Since $w\in \Dom(\Delta)$, $f \cdot (Uw)\in L^2(\Omega,d\mu)$, hence the integral in this last expression is finite.
It follows that $\alpha(\eps,r)\le C \eps |r|$. 
From (\ref{phiw}) and these calculations we therefore conclude that for some $C=C(\vphi,F)$,
\begin{equation} \label{phiwest}
|(\vphi w)^\wedge(\xi)| \le C\eps + |(\vphi \teps w)^\wedge(\xi)|.
\end{equation}
We now show that for any $m\in \N_0$, $|\xi|^{\frac{2m}{N+1}}|(\vphi w)^\wedge(\xi)|$ is bounded on $\Gamma$,
thereby demonstrating that $(x_0,\xi_0)\not\in \WF(w)$. 

Suppose to the contrary that there exists some $m\in \N_0$ and a sequence $\xi_j\in \Gamma$ with $|\xi_j|\to \infty$ 
such that $|\xi_j|^{\frac{2m}{N+1}}|(\vphi w)^\wedge(\xi_j)| \to \infty$ as $j\to \infty$. Then 
$\eps_j := |\xi_j|^{-\frac{2m}{N+1}} \to 0$, and using (\ref{phiwest}), we obtain
\begin{eqnarray*}
|\xi_j|^{\frac{2m}{N+1}}|(\vphi w)^\wedge(\xi_j)|&=&\eps_j^N |\xi_j|^{2m} |(\vphi w)^\wedge(\xi_j)| \\
&\le& C\eps_j^{N+1} |\xi_j|^{2m} + 
\eps_j^N |\xi_j|^{2m}|(\vphi \teps w)^\wedge(\xi)|
\end{eqnarray*}
By (\ref{nogwf}), however, the right hand side of this inequality is globally bounded, a contradiction.

Conversely, suppose that $(x_0,\xi_0)\not\in \WF(w)$. We have to show that $(x_0,\xi_0)\not\in \WFg(\iota_T(w))$. As 
above, we may without loss of generality suppose that $w\in L^2(M)$. Pick some open neighborhood
$U$ of $x_0$ and some conic neighborhood $\Gamma_1$ of $\xi_0$ in $\R^n\setminus 0$ such that 
$(U\times\Gamma_1)\cap \WF(w)=\emptyset$. Let us suppose for the moment that we already know that,
setting $u(s,x):=\cos(s\sqrt{\Delta})w$, we have  
\begin{equation}\label{wflemma}
\exists\, s_0>0: (U\times \Gamma_1)\cap \{(x,\xi) \mid \exists s,\ |s|\le s_0\ \exists \tau: (s,x;\tau,\xi)\in 
\WF(u)\} =\emptyset.
\end{equation}
Then, given $\chi\in \D(U)$ and $\nu\in \D((-s_0,s_0))$, for each $l\in \N$ there exists some $C_l>0$ such that
$$
|((\nu\otimes \chi)u)^{\wedge}(\tau,\xi)| \le C_l(1+|\tau|+|\xi|)^{-l} \quad (\tau\in \R,\ \xi\in \Gamma_1).
$$
Thus for $l>n$ and $|s|\le s_0$ we obtain
\begin{eqnarray*}
|\nu(s)(\chi\cdot u(s,\,.\,))^\wedge(\xi)| &=& |{\mathcal F}_{\tau \to s}^{-1}(((\nu\otimes \chi)u)^{\wedge}(\xi,\tau))|\\
&\le& \int_\R \frac{C_l\,d\tau}{(1+|\tau|+|\xi|)^l} = O((1+|\xi|)^{-l})
\end{eqnarray*}
In addition, we now choose $c$ such that $2c<s_0$ (which is possible by Lemma \ref{philemma}) and 
$\nu\in \D((-s_0,s_0))$ such that $\nu\equiv 1$ on $\supp \phi_c$. Then for $\xi\in \Gamma_1$,
\begin{eqnarray*}
|(\chi\cdot \iota_T(w)_\eps)^\wedge(\xi)| &=& |\int_\R \phi_c(s) (F_\eps)^\wedge(s)\nu(s)
\int_{\R^n} e^{-i\xi x} \chi(x) u(s,x)\,dx\, ds|\\
&=& O((1+|\xi|)^{-l}).
\end{eqnarray*}
Thus, $(x_0,\xi_0)\not\in \WFg(\iota_T(w))$, as claimed.

It remains to establish (\ref{wflemma}). To this end, denote by $\beta$ the bicharacteristic flow on
$T^*(\R\times M)$ corresponding to $\partial_s^2 + \Delta$. Since $u$ is the solution to (\ref{wave_eqn})--(\ref{waveic2})
with $u_0=w$, by \cite{Du}, p 118, $\WF(u) \subseteq C_0\circ \WF(w)$, where 
\begin{eqnarray*}
&&C_0 = \{((s,x;\tau,\xi),(x_0,\xi_0))\mid \exists r,\, \tau_0\in\R:(s,x;\tau,\xi) = \beta(r, (0,x_0,\tau_0,\xi_0)) \\
&& \hphantom{C_0 = \{((s,x;\tau,\xi),(x_0,\xi_0))\mid} \wedge \  -\tau_0^2 + g_{x_0}(\xi_0,\xi_0) = 0\}.
\end{eqnarray*}
$\beta$ is the flow of the Hamiltonian vector field of the symbol $-\tau^2+g_x(\xi,\xi)$, so the
corresponding system of ODEs reads
$$
\begin{array}{rcl}
\dot s (r) &=& -2\tau(r)\\
\dot x(r) &=& 2 g_{x(r)}(\xi(r),\,.\,)\\
\dot \tau(r) &=& 0\\
\dot \xi(r) &=& -Dg(x(r))(\xi(r),\xi(r))
\end{array}
$$ 
Denoting by $\beta_i$ the $i$-th component of $\beta$, it follows that
\begin{eqnarray*}
&&C_0 = \{((s,x;\tau,\xi),(x_0,\xi_0))\mid \tau\equiv \tau_0 = \pm \sqrt{g_{x_0}(\xi_0,\xi_0)},\, s=-2r\tau_0,
\exists r\in\R: \\
&& \hphantom{C_0 = \{((s,x;\tau,\xi),(x_0,\xi_0))\mid} 
(x,\xi) = (\beta_2,\beta_4)(r, (0,x_0,\tau_0,\xi_0))\},
\end{eqnarray*}
and, since $\WF(u) \subseteq C_0\circ \WF(w)$,
\begin{eqnarray*}
&&\WF(u) \subseteq \{(s,x;\tau_0,\xi) \mid  
\tau_0 = \pm \sqrt{g_{x_0}(\xi_0,\xi_0)},\, s=-2r\tau_0,
\exists r\in\R: \\
&& \hphantom{\WF(u) \subseteq \{ }
\exists (\bar x,\bar\xi)\in \WF(w)\, \exists r\in \R: (x,\xi) = (\beta_2,\beta_4)(r, (0,\bar x,\tau_0,\bar\xi)) \}.
\end{eqnarray*}
By continuity of $\beta$ and the fact that $(U\times\Gamma_1)\cap \WF(w)=\emptyset$, (\ref{wflemma}) follows.
\end{proof}

Summing up, we obtain
\begin{theo} \label{mainth}
The family $(\teps)_{\eps\in I}$ defined by (\ref{tepsdef}) is an optimal regularization process. The corresponding
embedding
\begin{eqnarray*}
\iota_T: \D'(M) &\to& \gs(M) \\
\iota_T(u) &=& [(\teps u)]
\end{eqnarray*}
is an injective sheaf morphism that renders $\cinfty(M)$ a subalgebra of $\gs(M)$. $\iota_T$ commutes
with isometries and pseudodifferential operators that commute with $\Delta$. Moreover, it preserves 
the singularity structure (wavefront set) of distributions.
\end{theo}
\begin{remark} We note that while Th.\ \ref{mainth} is formulated using the language of
algebras of generalized functions, it can also be used independently of this theory. For example,
on the level of regularizing nets, preservation of wavefront sets under $\iota_T$ means that
the wavefront set of a distribution $w\in \D'(M)$ can be read off from the asymptotic
properties of its regularization $(\teps(w))$ via (\ref{wfest}), and similar for the other
properties.
\end{remark}

\section{Distributional sections of a vector bundle} \label{vectorsec}
In this section we consider the problem of regularizing distributional sections of a vector bundle over a manifold.  We shall  provide a notion of optimal regularization and show that given a differential operator $D$ satisfying two simple conditions, we can always obtain such regularizations.

Let $|\Lambda| M$ denote the density bundle over $M$.  Then for $E^*$ the dual vector bundle  of some vector bundle $E$,
the space of distributional sections of $E$ is given by $\D'(M:E):=\Gamma_c^{\infty}(M:E^*\otimes |\Lambda| M)'$. In particular we have a natural inclusion $\Gamma^{\infty}(M:E)\rightarrow \D'(M:E)$. By  choosing a trivialization of the density bundle $|\Lambda| M$, for example by choosing a Riemannian metric, and by choosing a Hermitian inner product on $E$ to identify with $E^*$ we can  (non-canonically)  identify $\D'(M:E)$  with $\Gamma_c^{\infty}(M:E)'$.  In the sequel we shall assume that we are given a Riemannian metric on $M$ and a Hermitian inner product on $E$.  We similarly define the space $\mathcal{E}'(M:E)$ of compactly supported $E$-valued distributions.

By a smoothing operator on $E$ we shall mean an operator defined by a kernel in $\Gamma^{\infty}(M\times M:~\operatorname{End}(E)\otimes  \Lambda_R)$\footnote{To be precise, let $\pi_L,\pi_R:M\times M\rightarrow M $ be the left and right projections on $M$.  Then  $\operatorname{End(E)}:=\pi_R^*(E)^*\otimes \pi_L^*(E)$ and  $\Lambda_R=\pi_R^*|\Lambda|M$.}. Then if $T$ is a smoothing operator then $T:\mathcal{E}'(M:E)\rightarrow \Gamma^{\infty}(M:E)$.

Having fixed our notations we shall define optimal regularization processes
for distributional sections analogous to Def.\ \ref{abstract_embedding}.

 \begin{definition}\label{bundle_embedding}
   A parametrized family $(T_{\varepsilon})_{\eps\in I}$ of properly supported smoothing operators is
 called an optimal regularization process if
 \begin{enumerate}[(A)]
 \item \label{moderatevb}
 The regularization of any compactly supported distributional section $s\in\mathcal{E}'(M:E)$ is of moderate growth:
For any continuous seminorm $\rho$ on $\Gamma^{\infty}(M:E)$, there exists some integer $N$ such that
   \[
\rho(T_{\varepsilon}s) = O(\varepsilon^{N}) \qquad (\eps\to 0),
\]
\item\label{Identityvb} The net $(T_{\varepsilon})$ is an approximate identity: for each
$s\in \maE'(M:E)$, 

\[\underset{\varepsilon \rightarrow 0}{\operatorname{lim}}T_{\varepsilon}s=s\quad \textrm{in}\,\,\maD'(M:E).\]

\item\label{negligiblevb} If $u\in\Gamma_c^{\infty}(M:E)$ is a smooth compactly supported section of $E$ then
for all continuous seminorms $\rho$ and given any integer $m$,
\[
\rho(T_{\varepsilon}u-u) = O(\varepsilon^m).
\]

\item\label{singularityvb} 
The  induced map $\iota_T: {\mathcal E}'(M:E)\rightarrow  \maG(M:E)$  preserves  support, singular support and the wavefront set. In particular, 
\[\iota_T(\D'(M:E))\cap \maG^{\infty}(M:E)= \Gamma^{\infty}(M:E).\]

\end{enumerate}
\end{definition}

In the following section we shall  describe the precise requirements on a differential operator $D$ that would provide us with the functional calculus necessary for the construction of  an optimal regularization.
\subsection{Admissible operators}
Let $E\rightarrow M$ be a vector bundle over a complete Riemannian manifold $M$ provided with a Hermitian inner-product $\ip{~~}_E$. We shall denote by $L^2(M:E)$ the completion of the compactly supported sections $\Gamma_c^{\infty}(M:E)$ with respect  to the norm
\[\|s\|:=\int_M \ip{s(x),s(x)}_Edx\quad s\in\Gamma_c^{\infty}(M:E).\]
\begin{definition}
Let $D$ be a symmetric first order differential operator on $E$.  We shall assume that
\begin{enumerate}

\item The operator $D$ has finite speed of propagation, that is the norm of the principal symbol over the unit sphere is bounded by a constant $C_D$,
\[C_D=\operatorname{sup} \{\|\sigma_D(x,\xi)\|~|~ x\in M,\ \|\xi\|=1\}<\infty.\]
\item  The operator $D$ is elliptic.
\end{enumerate}
Such a differential operator $D$ shall be called admissible operator.
\end{definition}

 As a consequence of the finite speed of propagation, $D$ is essentially self-adjoint.  Therefore the equation
 \begin{eqnarray}\label{transport_eqn}
 \frac{\partial}{\partial t}u=iDu\qquad u(\,.\,,0)=u_0,
 \end{eqnarray}
 has a unique solution for all times $t$ for any initial datum $u_0\in\Gamma_c^{\infty}(M:E)$. Uniqueness follows from energy estimates, while existence is seen by applying functional calculus to note that $e^{itD}u_0$ is a solution.
 
 Furthermore  for a Schwartz function $F\in\Sch(\RR)$ the Fourier inversion formula  gives that
 \begin{eqnarray}\label{FIF}
  F(D)=\frac{1}{2\pi}\int_{-\infty}^{\infty}\hat{F}(s)e^{isD}ds.
 \end{eqnarray}
(with respect to the strong operator topology).
 
 As already noted, if $u$ is supported in a set $L$  then  $e^{itD}u$ is supported in the ball $B_{C_D\cdot t}(L)$. This has the following consequence:
 
\begin{lem}
If $F\in \Sch(\RR)$ is a Schwartz function  such that $\hat{F}$ is supported in an interval $(-c,c)$ then, for any $u\in L^2(M:E)$,
\[\supp( F(D)u)\subseteq B_{C_D\cdot c}(\supp(u)).\] 
\end{lem}

On the other hand the ellipticity of $D$ insures that the  operator defined by applying a Schwartz function $F$ to $D$ is necessarily a smoothing operator.

Let as before $F$ be an even Schwartz function in $\Sch(\RR)$ such that $F\equiv 1$ near the origin. Let $F_{\varepsilon}(x):=F(\varepsilon x)$. Our main result in this section is 
\begin{theo}\label{operator_embedding}
Given an admissible differential operator $D$ and a Schwartz function $F$ as above the family of operators $(F_{\varepsilon}(D))_{\eps\in I}$ provides an optimal regularization process in sense of Def.\ \ref{bundle_embedding}.
\end{theo}
The next two subsections provide the arguments for the proof.

\subsection{Weyl's law and functional calculus.}
In this subsection we shall assume that $M$ is compact.  In this case any symmetric operator $D$  is essentially self-adjoint. In addition the operator $D^2$ is a positive elliptic operator by assumption and hence Weyl's asymptotic formula for eigenvalues gives
\begin{eqnarray}\label{weyl}
N_{D^2}(\lambda):=\#\{\lambda_i\in \operatorname{sp}(D^2)|~~\lambda_i\leq \lambda\}\sim C
\lambda^{\frac{\operatorname{dim}(M)}{2}}.
\end{eqnarray}
 Then the following  can be obtained by applying \eqref{weyl}.
\begin{lem}\label{weyl_law}
Let $D$ be an elliptic self-adjoint differential operator of order $1$ and let $M$ be compact.  Then for a Schwartz function $F$ on $\R$ with $F\equiv 1$ near the origin we have:
\begin{enumerate}[(A)]
\item Given a smooth section $u\in \Gamma(M:E)$
\[\|F_{\varepsilon}(D)u-u\|_{L^2(M:E)} = O(\varepsilon^m)\quad \textrm {for all }~m\in\ZZ.\]

\item If $s\not\in H^k(M:E)$  for every $k>t$  then given any $\delta>0$,
$\|F_{\varepsilon}(D)s\|_{L^2(M:E)}$ is not $O(\varepsilon^{\frac{\operatorname{dim~M}}{2}+t+\delta})$.
In particular,
\[\iota_{F_{\varepsilon}(D)}(\D'(M:E))\cap \maG^{\infty}(M:E)= \Gamma^{\infty}(M:E)).\]

\item For every distributional section $s$ the regularization $(F_{\varepsilon}(D)s)$ is  moderate.
\end{enumerate}
\end{lem}
The proof of the  above  Lemma can be found in \cite{D}. This result is  precisely due to the fact that $D^2$ is a positive elliptic operator. We still need to prove that  the microlocal properties hold true for our regularizations $F_{\varepsilon}(D)$. These turn out  to be precisely due to the finite speed of propagation of $D$.

\subsection{Finite speed of propagation and localization}
We now return to the general situation where $M$ is a complete Riemannian manifold not necessarily compact.

Recall that if $X$ is a compact manifold with boundary then one can obtain a double of $X$, 
denoted here by $DX$ by  gluing two copies of $X$  along the boundary $\partial X$ (e.g., \cite{Ko}, VI 5.1).   
Now if $X$ is a compact manifold with boundary embedded in a Riemannian manifold $M$ of the same dimension and 
if $U$ is an open subset of $M$ such that $\bar{U}\subset \operatorname{interior}(X)$, then one can choose a 
Riemannian metric on $DX$ so that the inclusion $j:U\hookrightarrow  DX$ is an isometry.  Furthermore it is 
clear that given any vector bundle $E\rightarrow M$ there exists a vector bundle $E_X\rightarrow DX$ such that $E_X$ 
restricted to $U$ is canonically isomorphic to $E|_U$. At the same time there exists a   symmetric  elliptic operator 
$D_X$ on $E_X$ that matches up with $D$ on $U$.

We fix a compactly supported distributional section $u\in\mathcal{E}'(M:E)$ and a constant $c>0$.  Since $M$ is complete 
the open ball $U:=B_{2c\cdot C_D}(\operatorname{supp}(u))$ is relatively compact and is  contained in a compact manifold  
with boundary $X\subseteq M$. Now $u$ can be identified with a distributional section of a vector bundle $E_X\rightarrow  DX$.

\begin{prop} With assumptions  on $u,c$ and $F$ as above, let  us further assume that the Fourier transform  $\hat{F}(s)$ is supported in an interval $(-c,c)$. Then $F(D)u$ and $F(D_X)u$ are both supported in $U$ and 
\[F(D)u=F(D_X)u.\]
\end{prop}
\begin{proof}
 Since the operators $D$  and $D_X$  restricted to the open set $U$ coincide, the uniqueness of solutions to the equation \eqref{transport_eqn}  implies that $e^{isD}u$ and $e^{isD_X}u$   agree for $s\le c$.  The statement therefore follows from the Fourier Inversion Formula  \eqref{FIF}.
\end{proof}
 We can now finish the proof of our main result.
 \begin{proof}[Proof of Theorem \ref{operator_embedding}]
 First we note that given any  cutoff function $\phi(s)$ supported in an interval $(-c,c)$ such that $\phi\equiv 1 $ near the origin, and any compactly supported  distributional section  $u$,
  \begin{align*}
     [F_{\varepsilon}(D)u]&=[\frac{1}{2\pi}\int^{\infty}_{-\infty}\phi(s)\hat{F}_\eps(s)e^{isD}u\,ds]\\
     &=[j^*\left(\frac{1}{2\pi}\int^{\infty}_{-\infty}\phi(s)\hat{F}_\eps(s)e^{isD_X}u\,ds\right)]
   \end{align*}
 in $\maG(M:E)$. 
 
 With this observation it is clear that:
 \begin{enumerate}
  \item  All estimates  for Definition \ref{bundle_embedding} follow from Lemma \ref{weyl_law}.
  \item  The support of $u$ coincides with the generalized support of $[F_{\ep}(D)u]$.  This implies that the embedding extends  to a sheaf morphism $i_{F_{\varepsilon}}:\D'(M:E)\rightarrow  \maG(M:E)$.
  \item Since wave-front sets are defined locally, our  embedding $\iota_{F_{\varepsilon}}$  preserves wavefront sets by Th.\ \ref{wfth}.
 \end{enumerate}
 \end{proof}

\begin{remark}
From the proof one notices that a second order positive elliptic differential operator $\Delta$ on sections of $E$ also  provides us with an  optimal embedding $F_{\ep}(\Delta)$  provided that the solution operator to the wave equation  \eqref{wave_eqn}, namely $\cos(s\sqrt{\Delta})$ propagates  support  at a finite speed.  Thus in particular if $T^r_s(M)$ denotes the tensor bundle on $M$ and $g$ a complete Riemannian metric on $M$ the induced Laplace operator $\Delta^r_s$ provides an example of such an operator.
\end{remark}

\subsection{Isomorphisms between vector bundles}
Let $\phi:E_1\rightarrow E_2$ be an isomorphism of Hermitian vector bundles (preserving the inner product). Given any admissible  differential operator $D_1$  on sections of $E_1$  the  push-forward $D_2:=\phi D\phi^{-1}$  is  also an iso-spectral admissible differential operator. In particular, for any Schwartz function $F$ we have
\[F(D_2)=\phi F(D_1)\phi^{-1}.\]
 The extension of $\phi_*:\Gamma^{\infty}(M:E_1)\rightarrow \Gamma^{\infty}(M:E_2)$  to the generalized sections, $\phi_*:\maG(M:E_1)\rightarrow \maG(M:E_2)$ commutes with the geometrical  embeddings $F_{\varepsilon}(D_1)$ and $F_{\varepsilon}(D_2)$. 
 
Thus for example if $r_1+s_1=r_2+s_2$ then the Riemannian metric provides an isomorphism $g:T^{r_1}_{s_1}(M)\rightarrow T^{r_2}_{s_2}(M)$  that pushes $\Delta^{r_1}_{s_1}$ to  $\Delta^{r_2}_{s_2}$. Hence the corresponding functional calculus embedding commutes with the lowering or raising of indices.

\subsection*{Acknowledgements} We would like to thank the referee for her/his remarks that have
led to a number of improvements in the paper.


\end{document}